\documentclass[a4paper,11pt]{amsart}
\usepackage{amssymb,amsfonts,amsxtra,
mathrsfs,placeins,graphicx,verbatim,stmaryrd}
\usepackage[all]{xy}
\xyoption{line}
\usepackage{fullpage}
\usepackage{euscript}
\newtheorem{theorem}{Theorem}[section]
\newtheorem{cor}[theorem]{Corollary}
\newtheorem{lem}[theorem]{Lemma}
\newtheorem{prop}[theorem]{Proposition}

\theoremstyle{definition}

\newtheorem{example}[theorem]{Example}
\newtheorem{defi}[theorem]{Definition}
\newtheorem{rem}[theorem]{Remark}

\numberwithin{equation}{section} \DeclareMathOperator{\Hom}{Hom}
 \DeclareMathOperator{\Hoch}{Hoch}
 
 \DeclareMathOperator{\MC}{MC}
 \DeclareMathOperator{\Def}{Def}
\DeclareMathOperator{\Map}{Map} 
 
\newcommand{\noproof}{\begin{flushright} \ensuremath{\square}
\end{flushright}}
\def\ground{\mathbf k}

\def\g{\mathfrak g}
\def\h{\mathfrak h}
\def\MCmod{\mathscr {MC}}

\def\ra{\rightarrow}

\def\Z{\mathbb{Z}}

\def\id{\operatorname{id}}
\def\Der{\operatorname{Der}}

\def\ad{\operatorname{ad}}

\def\CE{\operatorname{CE}}

\def\Hoch#1{B#1}
\def\Hoch{\operatorname{Hoch}}
\def\CHoch{\operatorname{CHoch}}

\def\tp#1#2{#2\otimes#1}

\thanks{}

\begin{document}

\title[L-infinity maps and twistings...]{L-infinity maps and
twistings}
\author{J. Chuang  \and A.~Lazarev}
\thanks{The second author is partially supported by an EPSRC research
grant}
\address{Centre for Mathematical Science\\ City University London\\London
EC1V 0HB\\UK}
\email{J.Chuang@city.ac.uk}
\address{University of Leicester\\ Department of
Mathematics\\Leicester LE1 7RH, UK.}
\email{al179@leicester.ac.uk} 
\begin{abstract}
We give a construction of an L-infinity map from any L-infinity
algebra into its truncated Chevalley-Eilenberg complex as well as its
cyclic and A-infinity analogues. This map fits with the inclusion
into the full Chevalley-Eilenberg complex (or its respective
analogues) to form a homotopy fiber sequence of L-infinity algebras.
Applications to deformation theory and graph homology are given. We
employ the machinery of Maurer-Cartan functors in L-infinity and
A-infinity algebras and associated twistings which should be of
independent interest.
\end{abstract}

\maketitle

\section{introduction}
Let $V$ be a Lie algebra and $f:V\ra\Hom(V,V)$ be the map that
associates to $v\in V$ the endomorphism $f(v):=\ad(v):x\mapsto [v,x]$
for $x\in V$. Then $f$ is clearly  a Lie algebra homomorphism
$V\ra\Der(V,V)$ from $V$ into the Lie algebra of derivations of $V$.
If instead $V$ is an associative algebra, $f$ is still a
Lie algebra map from the commutator Lie algebra of $V$ into the Lie
algebra of derivations of $V$. Furthermore, a variant of
this construction works when $V$ has an invariant scalar
product $\langle, \rangle$, so that $\langle[a,b],c\rangle=\langle
a,[b,c]\rangle$ in the Lie case and $\langle ab,c\rangle=\langle
a,bc\rangle$ in the associative case; here $a,b,c\in V$. In this case
the image of $f$ lands in the subspace of $\Der(V,V)$ consisting of
\emph{skew self-adjoint} endomorphisms of $V$.

One of the purposes of this paper is to generalize this simple and well-known
construction to $L_\infty$- and $A_\infty$-algebras. This
generalization will be applied to characteristic classes of
$A_\infty$-algebras in a future paper, but we believe that it is of
independent interest. A homotopy-invariant analogue of $\Der(V,V)$ for an
$L_\infty$-algebra $V$ is a truncated Chevalley-Eilenberg complex
$\overline{\CE}(V,V)$ which is part of the following homotopy fiber
sequence of complexes:
\begin{equation}\label{fibL1}
\xymatrix{V\ar^-f[r]&\overline{\CE}(V,V)\ar^-g[r]&\CE(V,V),}
\end{equation}
where $\CE(V,V)$ is the usual (untruncated) Chevalley-Eilenberg
complex and $f$ is a certain generalization of the map
$f=\ad$ defined above. Similarly, in the $A_\infty$ algebra context, a homotopy-invariant analogue of $\Der(V,V)$ is a truncated Hochschild complex
$\overline{\Hoch}(V,V)$ which is part of the following homotopy fiber
sequence of complexes:
\begin{equation}\label{fibL2}
\xymatrix{V\ar^-f[r]&\overline{\Hoch}(V,V)\ar^-g[r]&\Hoch(V,V),}
\end{equation}
where $\Hoch(V,V)$ is the usual (untruncated) Hochschild complex.
Furthermore, in the presence of an invariant scalar product (which
gives rise to a cyclic, or symplectic, $A_\infty$- or
$L_\infty$-algebra $V$, cf. \cite{HL}) we have the following homotopy
fiber sequences:
\begin{equation}\label{fibL3}
\xymatrix{V\ar^-f[r]&\overline{\CE}(V)\ar^-g[r]&\CE(V),}
\end{equation}
in the $L_\infty$-case; here $\CE(V)$ is the usual (untruncated)
Chevalley-Eilenberg complex with trivial coefficients and
$\overline{\CE}(V)$ is its truncated version and
\begin{equation}\label{fibL4}
\xymatrix{V\ar^-f[r]&\overline{\CHoch}(V,V)\ar^-g[r]&\CHoch(V,V),}
\end{equation}
in the $A_\infty$-case; here $\CHoch(V)$ is the usual (untruncated)
cyclic Hochschild complex and $\overline{\CHoch}(V)$ is its truncated
version.

We show that in all four cases the map $f$ can be lifted to an
$L_\infty$-map $\mathbf{f}$ possessing a certain universal property which,
roughly, says that $\mathbf{f}$ takes any Maurer-Cartan element $\xi$ in $V$
into the Maurer-Cartan element in $\overline{\CE}(V,V)$, (or
$\overline{\Hoch}(V,V), \overline{\CE}(V), \overline{\CHoch}(V,V)$)
which corresponds to the $L_\infty$-structure on $V$ obtained by
\emph{twisting} with $\xi$. The notions of $\infty$-Maurer-Cartan
elements and twisted $\infty$-structures are well-known and appear in
many papers on homological algebra, rational homotopy theory, mirror
symmetry and deformation theory, cf. for example \cite{ Kontsevich,
Merkulov, Fukaya, Getzler}. It seems hard to find a paper collecting
all necessary facts about $\infty$-Maurer-Cartan elements and
twistings and so we give all definitions and proofs of the results
that we need. We do, however, make use of the technology of formal
noncommutative (symplectic) geometry devised by Kontsevich \cite{Kon}
and treated at length in \cite{HL} in the context of
infinity structures.

The obtained results are used to produce a map from the
Chevalley-Eilenberg homology of a cyclic $L_\infty$- or
$A_\infty$-algebra into the corresponding version of a graph complex.

Finally, we show that the sequences (\ref{fibL1}),
(\ref{fibL2}), (\ref{fibL3}), (\ref{fibL4}) are homotopy fiber
sequences of $L_\infty$-algebras (as opposed to simply complexes);
this makes a connection to a recent work of Fiorenza and Manetti
\cite{FM}, as well as an earlier work by Voronov on derived brackets
\cite{Vor, Vor'}.

The paper is organized as follows. The remaining part of the
introduction reviews some results and terminology from $L_\infty$-
and $A_\infty$-algebras and their cohomology theories. Section 2
introduces the notion of Maurer-Cartan space in the infinity-context
and establishes some basic properties needed later on; these results
are undoubtedly known to experts but difficult to find in the
literature. The construction of the $L_\infty$-map $\mathbf{f}$ and a relation
to the graph homology are given in Section 3 and the last section
establishes a connection to the Fiorenza-Manetti construction and
applications to deformation theory.
\subsection{Acknowledgement} The authors wish to thank T.
Schedler, for useful comments on a preliminary version of this paper, and
the referee, for many helpful corrections and suggestions.

\subsection{Notation and conventions}
In this  paper we work in the category of  $\Z/2$-graded
 vector spaces (also known as super-vector spaces) over a field
 $\ground$ of characteristic zero.
However all our results (with obvious modifications) continue to hold
in the $\Z$-graded context. The adjective `differential graded' will
mean `differential $\Z/2$-graded' and will be abbreviated as `dg'. A
(commutative) differential graded (Lie) algebra will be abbreviated
as (c)dg(l)a.  We will often invoke the notion of a \emph{formal}
(dg) vector space; this is just an inverse limit of
finite-dimensional vector spaces. An example of a formal space is
$V^*$, the $\ground$-linear dual to a discrete vector space~$V$. A
formal vector space comes equipped with a topology and whenever we
deal with a formal vector space all linear maps from or into it will
be assumed to be continuous; thus we will always have $V^{**}\cong
V$.  All of our unmarked tensors are understood to be taken over
$\ground$. If $V$ is a discrete space and $W=\lim_\leftarrow{W_i}$ is
a formal space we will write $V\otimes W$ for
$\lim_{\leftarrow}V\otimes W_i$; thus for two discrete spaces $V$ and
$U$ we have $\Hom(V,U)\cong U\otimes V^*$.
 For a $\Z/2$-graded vector space $V=V_0\oplus V_1$ the symbol $\Pi
 V$ will denote the \emph{parity reversion} of $V$; thus $(\Pi
 V)_0=V_1$ while $(\Pi V)_1=V_0$. The symbol $S_n$ stands for the symmetric group on $n$ letters; it is understood to act on the $n$th tensor power of a $\mathbb Z/2$-graded vector space $V^{\otimes n}$ by permuting the tensor factors.

\subsection{Recollections on $L_\infty$- and $A_\infty$-algebras}  Our
basic reference on $L_\infty$- and $A_\infty$-algebras and their
cohomology is \cite{HL}; we now recall some of the results and
terminology from that paper. A non-unital formal
(c)dga is an inverse limit of finite-dimensional nilpotent non-unital (c)dgas.
We will, however, always work with \emph{unital} cdgas, and a formal
cdga will then mean a cdga obtained by adjoining a unit to a formal
non-unital cdga; thus our formal cdgas will be augmented.  The
augmentation ideal of an augmented cdga $A$ will be denoted by $A_+$.
There is an unfortunate clash of terminology: our notion of formality
is \emph{different} to the corresponding notion in rational homotopy
theory. The main examples of formal dgas and cdgas will be completed
tensor and symmetric algebras on (possibly formal) graded vector
spaces $V$; these will be denoted by $\hat{T}V$ and $\hat{S}V$
respectively.

Morphisms between and derivations of formal algebras will not be required
to be compatible with augmentations. Note however that the prototypical examples above (and indeed most of the examples of formal algebras
we consider) are local, and hence continuous morphisms between
them automatically preserve augmentations.

An $A_\infty$-algebra structure on a graded vector
space $V$ is a continuous odd derivation $m$ of the formal dga
$\hat{T}\Pi V^*$ and an $L_\infty$-algebra structure on $V$ is a
continuous odd derivation $m$ of the formal cdga $\hat{S}\Pi V^*$;
additionally $m$ is required to square to zero and have no constant
term. The components $m_i:{T}^i\Pi V\ra \Pi V$ or ${S}^i\Pi V\ra \Pi
V$
of the dual of the restriction of $m$ to $\Pi V^*$
 are the structure maps of the corresponding $A_\infty$- or
 $L_\infty$-structure. If $(V,m)$ and $(V^\prime, m)$ are two
 $A_\infty$- or $L_\infty$-algebras then a morphism between them is a
 continuous (c)dga map $(\hat{T}\Pi V^{\prime*},m')
 \ra (\hat{T}\Pi V^*,m)$ or $(\hat{S}\Pi V^{\prime *},m')
 \ra (\hat{S}\Pi V^*,m)$, respectively.

The Hochschild complex  of an $A_\infty$-algebra $(V,m)$
is the dg vector space  \[\Hoch(V,V):=(\Der(\hat{T}\Pi V^*),[-,m])\] consisting of
(continuous) derivations of $\hat{T}\Pi V^*$ supplied with the
differential $d(x)=[x,m]$ for $x\in \hat{T}\Pi V^*$; this is clearly
a dgla. The \emph{truncated} Hochschild complex
$\overline{\Hoch}(V,V)$ is the sub-dgla of $\Hoch(V,V)$ consisting of
derivations with vanishing constant term.

Similarly the Chevalley-Eilenberg complex $\CE(V,V)$ of an
$L_\infty$-algebra $(V,m)$ is the dg vector space~  $(\Der(\hat{S}\Pi
V^*),[,m])$ consisting of (continuous) derivations of $\hat{S} \Pi V$
supplied with the differential $d(x)=[x,m]$ for $x\in \hat{S}\Pi
V^*$; this is clearly a dgla. The \emph{truncated}
Chevalley-Eilenberg complex $\overline{\CE}(V,V)$ is the sub-dgla of
$\CE(V,V)$ consisting of derivations with vanishing constant term.

We stress that the algebras $\Der(\hat{T}\Pi V^*)$ and
$\Der(\hat{S}\Pi V^*)$ of derivations depend only on the dg vector space $V$.
On the other hand the complexes $\Hoch(V,V)$ and $\CE(V,V)$ depend on $(V,m)$, the
dg vector space $V$ equipped with an $A_\infty$- or $L_\infty$-algebra structure.

We note here that it is more traditional to call the Hochschild
complex or the Chevalley-Eilenberg complex of $(V,m)$ the dg vector space
$\Pi\Hoch(V,V)$ or $\Pi\CE(V,V)$; this has the unpleasant effect that
the dgla structures on these spaces become odd which also introduces
additional signs in various formulas, and this is the reason for our
convention; a similar remark also applies to the cyclic complexes
below.

For a usual associative algebra the complex $\Hoch(V,V)$ has the
familiar form \[V\ra\Hom(V,V)\ra\ldots\ra\Hom(V^{\otimes
n},V)\ra\ldots,\] whereas the complex $\overline{\Hoch}(V,V)$ has the
same Hochschild differential but starts with the term $\Hom(V,V)$.
Similarly, the complex $\CE(V,V)$ for a Lie algebra $V$ has the
form
\[V\ra\Hom(V,V)\ra\ldots\ra\Hom(\Lambda^n(V),V)\ra\ldots,\]
whereas the complex $\overline{\CE}(V,V)$ has the same
Chevalley-Eilenberg  differential but starts with the term
$\Hom(V,V)$. It is clear that, generally, an element in
$\overline{\Hoch}(V,V)$ is a (possibly infinite) linear combination
of Hochschild cochains, i.e. multilinear maps
$(\Pi V)^{\otimes n}\ra \Pi V$,
 and an element in  $\overline{\CE}(V,V)$ is a (possible infinite)
 linear combination of Chevalley-Eilenberg cochains, i.e.
 graded-symmetric multilinear maps
$(\Pi V)^{\otimes n}\ra \Pi V$.

In the presence of a graded symmetric, non-degenerate inner product
$\langle,\rangle$ on $V$ we can define cyclic (also called
symplectic) $L_\infty$- and $A_\infty$-algebras by the requirement
that the expressions $\langle m_i(x_1, \ldots, x_i),x_{i+1}\rangle$
be invariant with respect to cyclic permutations of the elements
$x_1,\ldots, x_{i+1}\in \Pi V$. The corresponding complexes are denoted
by $\CHoch(V,V)$, $\overline{\CHoch}(V,V)$,  $\CE(V)$ and
$\overline{\CE}(V)$ respectively.

So, $\CHoch(V,V)$ is the
sub dg subspace of $\Hoch(V,V)$ consisting
of cyclically invariant Hochschild cochains,  and
$\overline{\CHoch}(V,V)$ is the corresponding
sub dg vector space of $\overline{\CHoch}(V,V)$.
Likewise, $\CE(V)$ is the
sub dg subspace of $\CE(V,V)$ consisting
of cyclically invariant Chevalley-Eilenberg cochains,  and
$\overline{\CE}(V)$ is the corresponding
sub dg vector space of
$\overline{\CE}(V,V)$.

Note that $\CE(V)$ admits an alternative
interpretation as follows. Consider the maximal ideal $\hat{S}\Pi
V^*_+\subset \hat{S}\Pi V^*$ consisting of formal power series with
vanishing constant term. Then the derivation $m$ restricts to
$\hat{S}\Pi V^*_+$ and the dg vector space $(\hat{S}\Pi V^*_+, m)$ is
isomorphic to $\CE(V)$; this is a consequence of the Poincar\'e lemma
in formal symplectic geometry. Note also that $\CE(V)$ is essentially
(disregarding the shift and the lowest degree term) the
Chevalley-Eilenberg complex of $(V,m)$ with \emph{trivial
coefficients}.

We will need to consider infinity-algebra structures over a formal
cdga $A$. We define an \emph{$A$-linear} $A_\infty$-algebra structure
on a graded vector space $V$ to be a continuous odd $A$-linear
derivation $m$ of the formal dga $\tp{\hat{T}\Pi V^*}{A}$; in
addition $m+d_A\otimes\id$ is required to square to zero and have vanishing
constant term. The components $m_i:\tp{{T}^i\Pi V}{A} \ra
\tp{\Pi V}{A}$ of the $A$-linear dual of the restriction of $m$ to
$\tp{\Pi V^*}{A}$
 are the structure maps of the corresponding $A_\infty$- or
 $L_\infty$-structure;
 they may be viewed as $A$-multilinear maps on $\tp{\Pi V}{A}$.
 Note that a ($\ground$-linear) $A_\infty$-algebra $(V,m)$ is also an
 $A$-linear
 $A_\infty$-algebra $(V,m^A)$, for any formal cdga $A$, by extension
 of scalars:
 $m^A:=\tp{m}{\id_A}$ (equivalently $m^A_i:=\tp{m_i}{\id_A}$).
   In a similar way we can introduce the notion of an
   \emph{$A$-linear} $L_\infty$-algebra structure on a graded vector
   space $V$. If, in addition, $V$ possesses a graded symmetric,
   non-degenerate inner product, $A$-linear cyclic $A_\infty$- and
   $L_\infty$-algebra structures on $V$ are defined in an obvious
   way.

\section{Maurer-Cartan functor and twisting in $L_\infty$- and
$A_\infty$-algebras}
In this section we describe the Maurer-Cartan functor for
$L_\infty$- and $A_\infty$-algebras in both cyclic and non-cyclic
cases and associated twisted structures.

\begin{defi}Let $V$ be a graded vector space.
\begin{enumerate}
\item
Let $m$ be an $L_\infty$-structure on $V$  and $A$ be a formal cdga
with maximal ideal $A_+$. Then an even element $\xi\in\tp{\Pi V}{A_+}$ is
\emph{Maurer-Cartan} if $(d_A\otimes\id)(\xi)+\sum_{i=1}^\infty
\frac{1}{i!}m^A_i(\xi^{\otimes i})=0$.
\item
Let $m$ be an $A_\infty$-structure on $V$  and $A$ be a formal dga
with maximal ideal $A_+$. Then an even element $\xi\in\tp{\Pi V}{A_+}$ is
\emph{Maurer-Cartan} if $(d_A\otimes\id)(\xi)+\sum_{i=1}^\infty m^A_i(\xi^{\otimes
i})=0$.
\end{enumerate}
The set of Maurer-Cartan elements in $\tp{\Pi V}{A_+}$ will be
denoted by $\MC(V,A)$. The correspondence $(V,A)\mapsto \MC(V,A)$ is
functorial in $A$.
\end{defi}
\begin{prop}\label{MCfunctor}\
\begin{enumerate}
\item
Let $(V,m)$ be an $L_\infty$-algebra and $A$ be a formal cdga. Then
functor $A\mapsto\MC(V,A) $ is represented by the formal cdga
$(\hat{S}\Pi V^*,m)$.
\item
Let $(V,m)$ be an $A_\infty$-algebra and $A$ be a formal dga. Then
functor $A\mapsto\MC(V,A) $ is represented by the formal dga
$(\hat{T}\Pi V^*,m)$.
\end{enumerate}
\end{prop}
\begin{proof} For (1) let $f:\hat{S}\Pi V^*\ra A$ be a map of formal
cdgas. Forgetting the differentials, such a map is equivalent to a
continuous linear map $\Pi V^*\ra A$, which is in turn determines an
element $\xi_f\in \tp{\Pi V}{A}$. The condition that $f$ commutes
with differentials could be rewritten as the commutative diagram:
\begin{equation}\label{comm}
\xymatrix{ \Pi V^*\ar^{f|_{\Pi V^*}}[d]\ar^{m^*}[r]& (S\Pi
V)^*\ar^\cong[r]& \hat{S}\Pi V^*
\ar^{f}[d]\\
A\ar^d[rr]&&A }
\end{equation}
Note that the canonical isomorphism
\[
(S^n\Pi V)^*:=((T^n\Pi V)_{S_n})^*=(\hat{T}^n\Pi V^*)^{S_n}\ra
(\hat{T}^n\Pi V^*)_{S_n}=:{S}^n\Pi V^*
\] identifying invariants with coinvariants has the form
\[(\hat{T}^n\Pi V^*)^{S_n}\ni x_1\otimes\ldots\otimes x_n\mapsto
\frac{1}{n!}x_1\otimes\ldots\otimes x_n\in (\hat{T}^n\Pi
V^*)_{S_n}.\]
Taking this into account we obtain that the commutativity of the
diagram (\ref{comm}) is equivalent to the Maurer-Cartan equation for
$\xi_f$:
\[ (d_A\otimes\id)(\xi_f)+  m_1^A(\xi_f)+\frac{1}{2!}m_2^{A}(\xi_f\otimes\xi_f)+\ldots=0\]
where $m_i^{A}:(\tp{\Pi V}{A})^{\otimes i}\ra \tp{\Pi V}{A}$ is the
$i$th higher bracket in the $A$-linear $L_\infty$-algebra $V$.

We omit the proof of (2) which is similar to, but simpler than, that
of (1) in that one does not to have to worry about identification of
$S_n$-invariants with $S_n$-coinvariants (and thus, no factorials
will appear in the formulas).
\end{proof}
\begin{rem}\label{nonform}
There is another version of the MC-functor which is occasionally
useful. For an $L_\infty$-algebra $(V,m)$ it associates to a \emph{not
necessarily formal} cdga $A$ the set $\xi\in\Pi V\otimes A$ such that
$\sum_{i=1}^\infty \frac{1}{i!}m^A_i(\xi^{\otimes i})=0$ and
similarly when $(V,m)$ is an $A_\infty$-algebra. The formal series above does not
have to converge and to justify this definition one has to impose
some additional restrictions on $V$, e.g. that it possesses only
finitely many non-zero higher products as is the case when $V$ is a
dga or a dgla. In this situation (at least when $V$ is
finite-dimensional) the functor $\MC(V,-)$ is still represented by a
cdga $(S\Pi V^*, m)$ in the $L_\infty$-case and $(T\Pi V^*, m)$ in
the $A_\infty$-case (notice the absence of the completion). All
statements in this section have obvious analogues in this context.
\end{rem}
\begin{rem}\label{A-L}
Let $(V,m)$ be an $A_\infty$-algebra. Viewing $m$ as a map $\Pi
V^*\ra\hat{T}\Pi V^*$ and composing with the projection $\hat{T}\Pi
V^*\ra \hat{S}\Pi V^*$ we obtain a map $\overline{m}:\Pi
V^*\ra\hat{S}\Pi V^*$ which, obviously, gives an $L_\infty$-structure
on $V$ (this is an infinity version of the familiar commutator Lie
algebra construction). Furthermore, for a formal cgda $A$ the
Maurer-Cartan equations in $\tp{\Pi V}{A}$ are the same regardless
of whether $(V,m)$ is considered as an $L_\infty$- or $A_\infty$-algebra;
in other words an element $\xi\in\tp{\Pi V}{A}$ is an
$L_\infty$-MC-element if and only if it is an $A_\infty$-MC-element.
\end{rem}
The Yoneda lemma now gives:
\begin{cor}\
\begin{enumerate}
\item Let $(V,m)$ and $(V^\prime, m^\prime)$ be two
$L_\infty$-structures and $A$ be a formal cdga. Then any
$L_\infty$-map $V\ra V^\prime$ determines and is determined by a
morphism of functors (in the variable $A$):
    $\MC(V,A)\ra\MC(V^\prime, A)$.
\item
Let $(V,m)$ and $(V^\prime, m^\prime)$ be two $A_\infty$-structures
and $A$ be a formal dga. Then any $A_\infty$-map $V\ra V^\prime$
determines and is determined by a morphism of functors (in the
variable $A$):
    $\MC(V,A)\ra\MC(V^\prime, A)$.

\end{enumerate}
\end{cor}
\begin{proof}
We only have to note that an $L_\infty$-map $V\ra V^\prime$ is by
definition a (continuous) map of formal cdgas $(\hat{S}\Pi V^{\prime
*}, m^\prime)\ra
 (\hat{S}\Pi V^*, m)$ and an $A_\infty$-map $V\ra V^\prime$ is  a
 (continuous) map of formal dgas $(\hat{T}\Pi V^{\prime
 *},m^\prime)\ra (\hat{T}\Pi V^*,m)$.
\end{proof}
We will now turn to twistings. Let $(V,m)$ be an $L_\infty$-algebra and
$A$ be a formal cdga. We shall explain how a Maurer-Cartan element $\xi\in\MC(V, A)$
gives rise to a new $L_\infty$-algebra $(V,m^{\xi})$ on the same underlying space; the
derivation $m^{\xi}$ is obtained by applying an appropriate automorphism of $\Der(\hat{T}\Pi V)$ to $m$, or
equivalently as a constant shift by $\xi$ in the arguments of $m$. Both points of view are useful and so  we describe both in detail and explain the relationship between them.

 We may regard
$\xi$ as a (formal) $A$-linear derivation of the formal algebra
$\tp{\hat{S}\Pi V^*}{A}$; indeed $\xi$ could be viewed as a linear
function on $\Pi V^*\cong \tp{\Pi(V^*)}{1}\subset
\tp{\hat{S}\Pi V^*}{A}$ and we extend it to the whole of $\tp{\hat{S}\Pi V^*}{A}$ by the $A$-linear Leibniz rule.

Consider now $e^\xi:=\id+\xi+\ldots$ as an automorphism of
$\tp{\hat{S}\Pi V^*}{A}$; it is clear that this formal power series
converges. This automorphism acts by conjugations on derivations of
$\tp{\hat{S}\Pi V^*}{A}$; in particular on those of square zero;
moreover it turns out that it takes one $A$-linear
$L_\infty$-structure on $V$ to another such structure; the precise
statement is given below.

A similar discussion is applicable to $A_\infty$-algebras as well.
Namely, let $(V,m)$ be an $A_\infty$-algebra, $A$ be a formal dga and
$\xi\in\MC(V, A)$. We shall view $\xi$ as a (formal) $A$-linear
derivation of the formal algebra $\tp{\hat{T}\Pi V^*}{A}$; indeed
$\xi$ could be viewed as a linear function on
$\Pi V^*\cong \tp{\Pi(V^*)}{1}\subset \tp{\hat{T}\Pi V^*}{A}$ and we extend it to
the whole of $\tp{\hat{T}\Pi V^*}{A}$ by the $A$-linear Leibniz
rule. Then the automorphism $e^\xi$ acts on the set of $A$-linear
$A_\infty$-structures on $V$.
Here is the general formulation:
\begin{theorem}\
\begin{enumerate}
\item
\begin{enumerate}
\item
Let $(V,m)$ be an $L_\infty$-algebra, $A$ be a formal cdga and
$\xi\in\tp{\Pi V}{A_+}$. Then $m^\xi:=e^\xi m^{A} e^{-\xi}-(d_A\otimes \id)(\xi)$,
considered as a derivation of $\tp{\hat{S}\Pi V^*}{A}$, has no
constant term (and thus, gives another $A$-linear
$L_\infty$-structure on $V$) if and only if $\xi$ is Maurer-Cartan.
The process of passing from $m$ to $m^\xi$ is called \emph{twisting}
by a Maurer-Cartan element $\xi$.
\item
For $x_1,\ldots, x_n\in\Pi V$ the following formula holds:
\begin{equation}\label{twist1}
m^\xi_n(x_1, \ldots, x_n)=\sum_{i=0}^\infty
\frac{1}{i!}m^A_{n+i}(\xi,\ldots, \xi, x_1,\ldots, x_n).
\end{equation}
\end{enumerate}
\item
\begin{enumerate}
\item
Let $(V,m)$ be an $A_\infty$-algebra, $A$ be a formal dga and
$\xi\in\tp{\Pi V}{A_+}$. Then $m^\xi:=e^\xi m^Ae^{-\xi}-(d_A\otimes \id)(\xi)$, considered
as a derivation of $\tp{\hat{T}\Pi V^*}{A}$, has no constant term
(and thus, gives another $A$-linear $A_\infty$-structure on $V$) if
and only if $\xi$ is Maurer-Cartan. The process of passing from $m$
to $m^\xi$ is called \emph{twisting} by a Maurer-Cartan element
$\xi$.
\item
For $x_1,\ldots, x_n\in\Pi V$ the following formula holds:
\begin{equation}\label{twist2}
m^\xi_n(x_1, \ldots, x_n)=\sum_{i=0}^\infty m^A_{n+i}(\xi,\ldots,
\xi| x_1,\ldots, x_n),
\end{equation}
where
$$m^A_{n+i}(\xi,\ldots, \xi| x_1,\ldots, x_n):=\sum m^A_{n+i}(z_1,\ldots,z_{n+i}),$$
the sum running over
all ${n+i}\choose{n}$ sequences $z_1,\ldots,z_{n+i}$ containing $x_1,\ldots, x_n$ in order, together
with $i$ copies of  $\xi$.
\end{enumerate}
\end{enumerate}
\end{theorem}
\begin{proof}
We will restrict ourselves to proving part (1), the proof for (2) is
similar but simpler. Choose a basis $x_i$ in $V$ and the
corresponding linear coordinates $x^i$ in $\Pi V^*$. Then we can
write $\xi=\sum_i\xi^ix_i$ where $\xi^i\in A$, and any element in
$\tp{\hat{S}\Pi V^*}{A}$ can be written as a formal power series
$f(x^1, \ldots, x^n,\ldots,)$ with coefficients in $A$. The action of
$e^\xi$ on $f$ is then given as $\xi\cdot f=f(x^1+\xi^1,\ldots,
x^n+\xi^n, \ldots)$.

Furthermore, for a derivation $\eta:=\sum f^i(x^1, \ldots,
x^n,\ldots,)\partial_{x^i}$ we have
\begin{equation}\label{conj}
e^{\xi}\eta e^{-\xi}=\sum_if^i(x^1+\xi^1,\ldots,
x^n+\xi^n,\ldots)\partial_{x^i}.
\end{equation}
Let us now take $\eta=m^A$ where $m_A$ is the $A$-linear derivation of
$\tp{\hat{S}\Pi V^*}{A}$ corresponding to a given $L_\infty$-structure on $V$.
We see that the derivation $e^{\xi}\eta e^{-\xi}-(d_A\otimes\id)(\xi)$ has no constant
term (and thus, determines another $L_\infty$-structure) if and only
if 
 $\sum_if^i(\xi^1,\ldots, \xi^n,\ldots)=(d_A\otimes\id)(\xi)$.

The right hand side of (\ref{conj}) is a sum of monomials of the form
\[f_{i_1\ldots i_n}(x^{i_1}+\xi^{i_1})(x^{i_2}+\xi^{i_1})\ldots
(x^{i_n}+\xi^{i_n})\partial_{x_i}\] where $f_{i_1,\ldots, i_n}\in A$.
Every such monomial represents a linear map
$\Pi V^*\ra \tp{\hat{S}\Pi V^*}{A}$ and the components of the dual map
$S^n\Pi V\ra\tp{\Pi V}{A}$ correspond to the matrix elements of $m^\xi$. It follows
that equation (\ref{twist1}) holds; moreover for $n=0$ the constant
part of the derivation $m^\xi$ corresponds to $m_0:
=(d_A\otimes\id)(\xi)+\sum_{i=0}^\infty \frac{1}{i!}m_{i}(\xi,\ldots, \xi)$ and so its
vanishing is indeed equivalent to the Maurer-Cartan condition as
claimed.
\end{proof}
\begin{example}\label{dgla}
One of the most important special cases of the construction above
concerns the case when $V$ is a dgla or dga (as opposed to a general
$L_\infty$- or $A_\infty$-algebra). In that situation the tensoring
with a formal cdga $A$ is unnecessary and we can consider
$\MC(V):=\{\xi\in \Pi V:d\xi+\frac{1}{2}[\xi,\xi]=0$\} (note that this
equation is the same in the associative situation since then
$\frac{1}{2}[\xi,\xi]=\xi^2$)  In this situation it is easy to check that
the corresponding twisted dga or dgla will have the same product as
$V$ and the twisted differential: $d^\xi(x)=dx+[\xi,x]$ for $x\in
V$.
\begin{rem}
Note that the formulas (\ref{twist1}) and (\ref{twist2}) do \emph{not} contain signs;
this happens because our $m_i$'s are viewed as multi-linear maps on $\Pi V$;
accordingly $\xi$ is treated as an \emph{even} element (or constant derivation) and so permuting it past other element does not give rise to signs.
\end{rem}
\end{example}
We have the following easy result; in order to formulate it we
observe that the definition of $\MC(V,A)$ extends in an obvious way
to any
 $A$-linear $L_\infty$- or $A_\infty$-algebra $V$.
Here and sometimes later on we will use the shorthand $V^\xi$ (or
$(\tp{V}{A})^\xi$) to denote the $A$-linear $L_\infty$- or
$A_\infty$-algebra obtained by twisting by a Maurer-Cartan element
$\xi$.
\begin{prop}\label{shift}\
\begin{enumerate}
\item
Let $V$ be an $L_\infty$-algebra, $A$ be a formal cdga and
$\xi\in\MC(V, A)$. Then for $\eta\in \MC(V, A)$ we have
$\eta-\xi\in\MC(V^\xi, A)$ and there is a one-to-one correspondence
$\MC(V, A)\ra\MC(V^\xi, A)$ given by $\eta\mapsto\eta-\xi$.
\item
Let $V$ be an $A_\infty$-algebra, $A$ be a formal dga and
$\xi\in\MC(V, A)$. Then for $\eta\in \MC(V, A)$ we have
$\eta-\xi\in\MC(V^\xi, A)$ and there is a one-to-one correspondence
$\MC(V, A)\ra\MC(V^\xi, A)$ given by $\eta\mapsto\eta-\xi$.
\end{enumerate}
\end{prop}
\begin{proof} We will only give a proof for (1) since it will also
work for (2) without any changes.
Denote by $m$ the given $L_\infty$-structure on $V$. If $\eta\in
\MC(V,A)$ then
\begin{equation*}
e^{\eta-\xi}m^\xi e^{-\eta+\xi}=e^{\eta-\xi}e^\xi m e^{-\xi}
e^{-\eta+\xi}=e^{\eta} m  e^{-\eta}.
\end{equation*}
Therefore $e^{\eta-\xi}m^\xi e^{-\eta+\xi} - d_A\otimes \id$ has no constant term and
${\eta-\xi}\in \MC(V^\xi, A)$. This argument is clearly reversible so
${\eta-\xi}\in \MC(V^\xi, A)$ if and only if $\eta\in \MC(V, A)$.
\end{proof}
\begin{rem}\label{nonform2}
In the next section we will need the analogue of the above result
when $V$ is a dgla and $A$ is a not necessarily formal (see Remark
\ref{nonform}). The proof will be the same, verbatim.
\end{rem}

The twisted infinity structures are compatible with invariant scalar
products in the following sense.
\begin{prop}\label{cyclictwist}\
\begin{enumerate}
\item
Let $(V,m)$ be a cyclic $L_\infty$-algebra, $A$ be a formal cdga and
$\xi\in\MC(V, A)$. Then $(V, m^\xi)$ is
an $A$-linear
 cyclic $L_\infty$-algebra.
\item
Let $(V,m)$ be a cyclic $A_\infty$-algebra, $A$ be a formal dga and
$\xi\in\MC(V, A)$. Then $(V, m^\xi)$ is
an $A$-linear cyclic $A_\infty$-algebra.
\end{enumerate}
\end{prop}
\begin{proof}
The proof is the same for both cases; let us therefore consider the
$L_\infty$-case. Recall that $\hat{S}\Pi(\tp{V}{A})^*$ is a formal
$A$-supermanifold with a symplectic structure determined by a given
scalar product on $V$ and that $m$ is a symplectic vector field. The
adjoint action of the constant vector field $\xi$ is a translation;
therefore it preserves symplectic vector fields. It follows that
$m^\xi$ is symplectic as required.

Alternatively, one can check directly, using formulas (\ref{twist1})
that $m^\xi$ has the desired cyclic invariance property.
\end{proof}

\subsection{Maurer-Cartan elements in Hochschild and
Chevalley-Eilenberg complexes}\label{MCsection}

Let $V$ be a graded vector space and $A$ be a formal cdga. Recall that an
$A$-linear $L_\infty$-structure on $V$ is an odd (continuous)
$A$-linear derivation $m$ of the formal cdga $\tp{\hat{S}\Pi V^*}{A}$ without
constant term and of square zero. Put another way, $m\in\MC(\g,A)$, where
$\g=\Der_0(\hat{S}\Pi V^*)$, the graded Lie algebra of
continuous derivations
 of $\hat{S}\Pi V^*$ without constant term.
Let us
 consider the corresponding
 $A$-linear dgla $(\tp{\g}{A})^m$ obtained by twisting the
 graded Lie
 algebra $\tp{\g}{A}$ by $m$;
  recall that the twisted differential will have the form
  $d(x)=[m,x]$ for $x\in \tp{\g}{A}$. By definition, the dgla
  $(\tp{\g}{A})^m$ is isomorphic to
  the truncated Chevalley-Eilenberg complex
   $\overline{\CE}_A(V,V)=\tp{\overline{\CE}(V,V)}{A}$ of the
   $A$-linear $L$-infinity algebra $(V,m)$.

Similarly, an $A$-linear $A_\infty$-algebra structure on $V$ is a
Maurer-Cartan element $m$ in the graded Lie algebra $\tp{\h}{A}$,
where $\h=\Der_0(\hat{T}\Pi V^*)$ is the graded Lie algebra of  continuous derivations of
the formal dga $\hat{T}\Pi V^*$ with vanishing constant term; the
twisted dgla $(\tp{\h}{A})^m$ is just the truncated Hochschild
complex $\overline{\Hoch}_A(V,V)=\tp{\overline{\Hoch}(V,V)}{A}$ of
$(V,m)$.

These observations have cyclic versions: suppose that $(V,m)$ is an
$A$-linear cyclic $L_\infty$-algebra. That means that the space $V^*$
is supplied with a (graded) symplectic structure and that $m$ is a
symplectic derivation. Consider the Lie subalgebra $\tilde{\g}$ of
$\g$ consisting of symplectic derivations; then
$m\in\MC(\tilde{\g},A)$ and  the twisted dgla
$(\tp{\tilde{\g}}{A})^m$ is identified with the truncated cyclic complex
$\overline{\CE}_A(V)=\tp{\overline{\CE}(V)}{A}$ of $(V,m)$.

Similarly if $(V,m)$ is a $A$-linear cyclic $A_\infty$-algebra then
the space $V^*$ is supplied with a (graded) symplectic structure and
$m$ is a (noncommutative) symplectic derivation. Consider the Lie
subalgebra $\tilde{\h}$ of $\h$ consisting of symplectic derivations;
then $m\in\MC(\tilde{\g},A)$ and the twisted dgla $(\tp{\tilde{\g}}{A})^m$ is identified with the truncated cyclic complex
$\overline{\CHoch}_A(V,V)=\tp{\overline{\CHoch}(V,V)}{A}$ of
$(V,m)$.

Then we have the following result which is a direct consequence of
Proposition \ref{shift}.

\begin{prop}\label{propshift}Let $A$ be a cdga:
\begin{enumerate}
\item
\begin{enumerate}
\item
Let $(V,m)$ be an $L_\infty$-algebra. Then there is a one-to-one
correspondence between the set of $A$-linear $L_\infty$-structures on $V$ and $\MC(\overline{\CE}(V,V),A)$; namely any $A$-linear
$L_\infty$-structure $m^\prime$ determines an MC-element
$m^\prime-m\in\MC(\overline{\CE}(V,V),A)$.
\item
Let $(V,m)$ be a cyclic $L_\infty$-algebra. Then~ there~~ is a
one-to-one correspondence between the set of $A$-linear cyclic
$L_\infty$-structures on
 $V$ and $\MC(\overline{\CE}(V),A)$; namely
any $A$-linear cyclic $L_\infty$-structure $m^\prime$ determines an
MC-element $m^\prime-m\in\MC(\overline{\CE}(V),A)$.

\end{enumerate}
\item
\begin{enumerate}
\item
Let $(V,m)$ be an $A_\infty$-algebra. Then there is a one-to-one
correspondence between the set of $A$-linear $A_\infty$-structures on
$V$ and $\MC(\overline{\Hoch}(V,V),A)$; namely any $A$-linear
$A_\infty$-structure $m^\prime$ determines an MC-element
$m^\prime-m\in\MC(\overline{\Hoch}(V,V),A)$.
\item
Let $(V,m)$ be a cyclic $A_\infty$-algebra. Then there is a
one-to-one correspondence between the set of $A$-linear cyclic
$A_\infty$-structures  on
 $V$ and $\MC(\overline{\CHoch}(V,V),A)$;
namely any $A$-linear cyclic $A_\infty$-structure $m^\prime$
determines an MC-element $m^\prime-m\in\MC(\overline{\CHoch}(V,V),
A)$.
\end{enumerate}
\end{enumerate}
\end{prop}
\noproof
\begin{rem}\label{nonform3}
There is an obvious analogue of the above result in the situation
when $A$ is a not necessarily formal cdga, cf. Remark \ref{nonform2}.
\end{rem}

\section{Main construction}
In this section we describe an $L_\infty$-map ${\mathbf
f}=(f_1,f_2,\ldots)$ from an $L_\infty$- or $A_\infty$-algebra to its
truncated  Chevalley-Eilenberg or Hochschild complex as well as the analogues
in the cyclic situation. The first component $f_1$ is the map $f$ appearing in the homotopy fiber sequences (1.1)-(1.4).

\begin{theorem}\label{Lmap}\
\begin{enumerate}
\item
\begin{enumerate}
\item
Let $(V,m)$ be an $L_\infty$-algebra.
Then the maps $f_r:(\Pi
V)^{\otimes r}\ra \Pi \overline{\CE}(V,V)$ given by the formulas
\begin{equation}\label{Linf1}
f_r(w_1,\ldots, w_r)(x_1,\ldots, x_n)= m_{r+n}(w_1,\ldots,
w_r,x_1,\ldots, x_n), \qquad (w_i, x_i\in \Pi V).
\end{equation}
define an $L_\infty$-map $\mathbf f: V\to \overline{\CE}(V,V)$.
The map $\mathbf f$ is characterized by the property
that for any formal algebra $A$ an MC-element $\xi\in\MC(V,A)$ is
mapped to an MC-element in $\MC(\overline{\CE}(V,V),A)$
corresponding to the twisting of $\tp{V}{A}$ by the element $\xi$.
\item
Let $(V,m)$ be a cyclic $L_\infty$-algebra.
Then the maps $f_r:(\Pi V)^{\otimes r}\ra \Pi \overline{\CE}(V)$ are given by the formulas (\ref{Linf1})
define an $L_\infty$-map $\mathbf f: V\to \overline{\CE}(V)$.
 The map $\mathbf f$ is characterized by the property
that for any formal algebra $A$ an MC-element $\xi\in\MC(V,A)$ is
mapped to an MC-element in $\MC(\overline{\CE}(V),A)$
corresponding to the twisting of $\tp{V}{A}$ by the element $\xi$.
\end{enumerate}
\item
\begin{enumerate}
\item
Let $(V,m)$ be an $A_\infty$-algebra.
Then the maps
$f_r:(\Pi V)^{\otimes r}\ra \Pi \overline{\Hoch}(V,V)$ given by
the formulas
\begin{equation}\label{Linf2}
f_r(w_1,\ldots, w_r)(x_1,\ldots, x_n)= \sum_{\sigma\in S_r}m_{r+n}(\sigma(w_1\otimes\ldots\otimes
w_r)|x_1,\ldots, x_n).
\end{equation}
define an $L_\infty$-map $\mathbf f: V\to \overline{\Hoch}(V,V)$. Here
$m_{r+n}(w_{i_1},\ldots,w_{i_r}|x_1,\ldots,x_n):=\sum \pm
m_{r+n}(z_1,\ldots,z_{r+n})$
where the sum ranges over all ${r+n}\choose{s}$ shuffles of the
sequences $w_{i_1},\ldots,w_{i_r}$ and $x_1,\ldots,x_n$, and the sign is
determined by the Koszul sign rule.
 The map $\mathbf f$ is characterized by
the property that for any formal algebra $A$ an MC-element
$\xi\in\MC(V,A)$ is mapped to an MC-element in
$\MC(\overline{\Hoch}(V,V),A)$ corresponding to the twisting of
$\tp{V}{A}$ by the element $\xi$.
\item
Let $(V,m)$ be an cyclic $A_\infty$-algebra.
Then the maps
$f_r:(\Pi V)^{\otimes r}\ra \Pi \overline{\CHoch}(V,V)$ given by
the formulas (\ref{Linf2})
define an $L_\infty$-map $\mathbf f: V\to \overline{\CHoch}(V,V)$.
 The map $\mathbf f$ is characterized by the
property that for any formal algebra $A$ an MC-element
$\xi\in\MC(V,A)$ is mapped to an MC-element in
$\MC(\overline{\CHoch}(V,V),A)$ corresponding to the twisting of
$\tp{V}{A}$ by the element $\xi$.
\end{enumerate}
\end{enumerate}
\end{theorem}
\noproof
\begin{rem}
The maps $\mathbf f$ in the $L_\infty$- and
$A_\infty$-contexts are compatible in the following sense.

Let $(V,m)$ be an $A_\infty$-algebra; it could be regarded as an
$L_\infty$-algebra, cf. Remark \ref{A-L}. Furthermore, extending this
observation, one sees that there are maps of dglas $\Hoch(V,V)\ra
\CE(V,V)$ and $\overline{\Hoch}(V,V)\ra\overline{\CE}(V,V)$. Then
straightforward inspection shows that there is a commutative diagram
of $L_\infty$-algebras and maps:
\[
\xymatrix{
V\ar^-{\mathbf f}[r]\ar^\cong[d]&\overline{\Hoch}(V,V)\ar[d]\ar[r]&\Hoch(V,V)\ar[d]\\
V\ar^-{\mathbf f}[r]&\overline{\CE}(V,V)\ar[r]&\CE(V,V) }
\]
Moreover, a similar statement also holds in the cyclic context.
\end{rem}

\subsection{Proof of Theorem \ref{Lmap}: $L_\infty$-algebra case}
Let $(V,m)$ be an $L_\infty$-structure on a graded vector space $V$.
To construct an $L_\infty$-map $V\ra\overline{\CE}(V,V)$ it is
sufficient to exhibit a map
\[
\MC(V,A)\rightarrow \MC(\overline{\CE}(V,V),A)
\]
for any formal cdga $A$, which is functorial in $A$. Indeed, the
functor $A\mapsto \MC(V,A)$ is represented by the cdga $(\hat{S}\Pi
V^*,m)$ and thus, by Yoneda's lemma such a functorial map is induced
by a cdga map $\hat{S}\Pi \overline{\CE}(V,V)^*\rightarrow \hat{S}\Pi
V^*$ i.e. the desired $L_\infty$-map.

Let $\xi\in\MC(V,A)$; it gives rise via twisting to the $A$-linear
$L_\infty$-algebra $V^\xi$ and hence, by Proposition~\ref{propshift}
(more precisely, its version described in Remark \ref{nonform}), to
an MC element $m^\xi-m\in\MC(\overline{\CE}(V,V),A)$.
Note that while in general neither $m$ nor $m^\xi$ is
contained in $A_+\otimes\Pi\overline{\CE}(V,V)$, their difference is.
  So the
correspondence $\xi\mapsto m^\xi-m$ defines the desired map
$\MC(V,A)\rightarrow \MC(\overline{\CE}(V,V),A)$.

We will now describe the map $\mathbf f$ explicitly. To do this we
perform the above argument for the universal example $A=(\hat{S}\Pi
V^*,m)$. Consider the canonical MC-element
 $\xi=\sum \tp{w_i}{w^i} \in \tp{\Pi V}{\hat{S}\Pi V^*}$, where
 $\{w_i\}$ is a basis of $\Pi V$
 and $\{w^i\}$ its dual basis in $\Pi V^*\subset \hat{S}\Pi V^*$. The
 corresponding twisted $A$-linear
 $L_\infty$-structure on $\tp{V}{A}$ is given by
\begin{eqnarray}
\nonumber m^\xi_n(x_1,\ldots,x_n) & = &
\sum_{r=0}^{\infty} \frac{1}{r!} m^{\hat{S}\Pi
V^*}_{r+n}\left(\underbrace{\xi,\ldots,\xi}_r,x_1,\ldots,x_n\right)
\\
\nonumber & = &
\sum_{r=0}^{\infty}\frac{1}{r!} m^{\hat{S}\Pi
V^*}_{r+n}\left(\sum_{i_1} \tp{w_{i_1}}{w^{i_1}},\sum_{i_2}
\tp{w_{i_2}}{w^{i_2}},\ldots, \sum_{i_r} \tp{w_{i_r}}{w^{i_r}},
x_1,\ldots,x_n\right) \\
&=& \label{sum} \sum_{\stackrel {r\geq 0}
{i_1,\ldots,i_r}}\frac{1}{r!}
\tp{m_{r+n}(w_{i_1},\ldots,w_{i_r},x_1,\ldots,x_n)}{w^{i_r}\ldots
w^{i_1}},
\end{eqnarray}
for $x_1,\ldots,x_n\in \Pi V$.

This structure is interpreted as an MC-element $\tilde\xi =
(\tilde\xi_n)$ in the dgla
\[
\tp{\overline{\CE}(V,V)}{\hat{S}\Pi V^*} = \bigoplus_{n\geq 1}
\tp{\Hom(S^n \Pi V, \Pi V)}{\hat{S}\Pi V^*},
\]
where
\[
\tilde\xi_n(x_1,\ldots, x_n) = \sum_{\stackrel {r\geq 1}
{i_1,\ldots,i_r}}\frac{1}{r!}
\tp{m_{r+n}(w_{i_1},\ldots,w_{i_r},x_1,\ldots,x_n)}
{w^{i_r}\ldots w^{i_1}}.
\]
(Note that the only the $r>0$ terms of the sum (\ref{sum}) contribute
here.) Finally, interpreting the MC-element $\tilde\xi\in
\tp{\overline{\CE}(V,V)}{\hat{S}\Pi V^*}$ as an $L_\infty$-map
${\mathbf f}=(f_r)$ from $V$ into
\[
\overline{\CE}(V,V)= \bigoplus_{n\geq 1}\Hom(S^n\Pi V,\Pi V),
\]
 we obtain the desired formula (\ref{Linf1}).
The factorial disappears in (\ref{Linf1}) for the same reason that it appears in the
proof of part (1) of Proposition~\ref{MCfunctor}.

Now assume that $V$ is a \emph{cyclic}
$L_\infty$-algebra. Then the above $L_\infty$-map
$V\ra\overline{\CE}(V,V)$  factors
through $\overline{\CE}(V)$.
This follows directly from the fact that a Maurer-Cartan twisting of
a cyclic $L_\infty$-algebra is again cyclic, cf.
Proposition \ref{cyclictwist}.

\subsection{Proof of Theorem \ref{Lmap}: $A_\infty$-algebra case}
Now we turn to the closely related associative case. Let $(V,m)$ be
an $A_\infty$-algebra. We will construct an $L_\infty$-map from the
associated $L_\infty$-algebra $(V,\overline{m})$ (cf.
Remark~\ref{A-L})
 to the dgla $\overline{\Hoch}(V,V)$.

To construct such a map it suffices to
 exhibit a map
\[
\MC(V,A)\rightarrow \MC(\overline{\Hoch}(V,V),A)
\]
for any formal cdga $A$, which is functorial in $A$. Indeed, the
functor $A\mapsto \MC(V,A)$ is represented by the cdga $(\hat{S}\Pi
V^*,\overline{m})$ and thus, by Yoneda's lemma, such a functorial map
is induced by a cdga map $\hat{S}\Pi
\overline{\Hoch}(V,V)^*\rightarrow \hat{S}\Pi V^*$ i.e. the desired
$L_\infty$-map.

Let $\xi\in\MC(V,A)$; it gives rise to a twisted $A$-linear
$A_\infty$-algebra $(\tp{V}{A})_\xi$. Hence it determines the
MC-element $m^\xi-m$ in the dgla $\tp{\overline{\Hoch}(V,V)}{A}$.
 So the correspondence $\xi\mapsto m^\xi - m$
defines the desired map $\MC(V,A)\rightarrow
\MC(\overline{\Hoch}(V,V),A).$

We will now describe the map $\mathbf f:=(f_1,f_2,\ldots)$
explicitly.
Consider the canonical MC-element
 $\xi=\sum \tp{w_i}{w^i} \in \tp{\Pi V}{\hat{S}\Pi V^*}$, where
 $\{w_i\}$ is a basis of $\Pi V$
 and $\{w^i\}$ its dual basis in $\Pi V^*\subset \hat{S}\Pi V^*$. The
 corresponding twisted $A$-linear
 $A_\infty$-structure on $\tp{V}{\hat{S}\Pi V^*}$ is given by the
 product
\begin{eqnarray}
\nonumber
 m_n^\xi(x_1,\ldots,x_n)
& = &
\sum_{r=0}^{\infty} m^{\hat{S}\Pi
V^*}_{r+n}\left(\underbrace{\xi,\ldots,\xi}_r|x_1,\ldots,x_n\right)
\\
\nonumber & = &
\sum_{r=0}^{\infty}m^{\hat{S}\Pi V^*}_{r+n}\left(\sum_{i_1}
\tp{w_{i_1}}{w^{i_1}},\sum_{i_2} \tp{w_{i_2}}{w^{i_2}},\ldots,
\sum_{i_r} \tp{w_{i_r}}{w^{i_r}}|x_1,\ldots,x_n\right) \\
&=& \label{sum1} \sum_{\stackrel {r\geq 0} {i_1,\ldots,i_r}}
\tp{m_{r+n}\left(w_{i_1},\ldots,w_{i_r}|x_1,\ldots,x_n\right)}
{w^{i_r}\ldots w^{i_1}},
\end{eqnarray}
for $x_1,\ldots,x_n\in\Pi V$.
This structure is interpreted as a MC-element $\tilde\xi =
(\tilde\xi_n)$ in the dgla
\[
\tp{\overline{\Hoch}(V,V)}{\hat{S}\Pi V^*} = \bigoplus_{n\geq
1}\tp{\Hom(T^n\Pi V, \Pi V)}{\hat{S}\Pi V^*},
\]
where
\[
\tilde\xi_n(x_1,\ldots, x_n) = \sum_{\stackrel {r\geq 1}
{i_1,\ldots,i_r}}
\tp{m_{r+n}(w_{i_1},\ldots,w_{i_r}|x_1,\ldots,x_n)}{w^{i_r}\ldots
w^{i_1}}.
\]
(Note that the only the $r>0$ terms of the sum (\ref{sum1})
contribute here.) Finally, interpreting the MC-element
$\tilde\xi\in\tp{\overline{\Hoch}(V,V)}{\hat{S}\Pi V^*}$ as an
$L_\infty$-map ${\mathbf f}=(f_r)$ from $V$ into
\[
\overline{\Hoch}(V,V)= \bigoplus_{n\geq 1}\Hom(T^n\Pi V,\Pi V),
\]
 we obtain the desired formula (\ref{Linf2}).

Now assume that $V$ is a \emph{cyclic} $A_\infty$-algebra. Then the above $L_\infty$-map
$V\ra\overline{\Hoch(V,V)}$ factors
through $\overline{\CHoch}(V,V)$, since a Maurer-Cartan twisting of
a cyclic $A_\infty$-algebra is again cyclic.

\subsection{Application to graph homology}
We now outline one application of the constructions in this section
to graph homology. Let $(V,m)$ be an $L_\infty$-algebra which we
assume to be finite-dimensional for simplicity and consider, as in Section~\ref{MCsection}, the
graded Lie algebra $\g=\Der_0(\hat{S}\Pi V^*)$ consisting of
(continuous) derivations of $\hat{S}\Pi V^*$ vanishing at zero. Note
that $\g$ has the same underlying space as $\overline{\CE}(V,V)$ but
\emph{vanishing differential}; particularly it does not depend on the
$L_\infty$-structure on $V$. Consider the universal twisting of $V$;
the $L_\infty$-structure on $\hat{S}\Pi V^*\otimes V$ corresponding
to the twisting by the canonical element. This $L_\infty$-structure
can  itself be viewed as an element in $\MC(\g,(\hat{S}\Pi V^*, m))$.
Note, however, that this element does not lie in $(\hat{S}\Pi
V^*)_+\otimes \Pi\g$ and thus, we are in the framework of Remark
\ref{nonform}.

It follows that there is a map of from the CE complex of $\g$ to that
of $V$. At this stage we pass to the \emph{dual homological
complexes} $\CE_*(\g)$ and $\CE_*(V)$; thus $\CE_*(V)\cong S\Pi V$
with the differential dual to $[-, m]$ and $\CE_*(\g)$ is the
standard CE homological complex of the graded Lie algebra $\g$.
The switch to the dual language is not necessary, strictly speaking, but
it is more natural for the interpretation of the result in terms of
graph homology.

 Note
that both $\CE_*(V)$ and $\CE_*(\g)$ are dg coalgebras. It follows
that there is a map of dg coalgebras: ${\chi}:\CE_*(V)\ra\CE_*(\g)$.
Since the coalgebra $\CE_*(\g)$ is cogenerated by the component $\Pi
\g$ it follows that such a map is determined by the composition
$\CE_*(V)\ra\CE(\g)\ra\Pi\g$; abusing the notation we will  denote
the latter map also by $\chi$. Furthermore, $\chi$ can be viewed as a
collection $\chi=(\chi_1\ldots,\chi_r\ldots)$ where $\chi_r:(\Pi
V)^{\otimes r}\ra\Pi \g$. Moreover, since $\g$ could be identified
with $\prod_{n=1}^\infty\Hom(S^n V,V)$ the element
$\chi_r(v_1,\ldots, v_r)$ is determined by its value on every
collection $x_1,\ldots x_n\in \Pi V$ for $n=0,1,\ldots$.

Clearly, the same arguments also work in the $A_\infty$ situation,
and also in the cyclic versions of both.   We have the following
result whose proof is essentially the same as that of Theorem
\ref{Lmap}:
\begin{prop}\label{curved}\
\begin{enumerate}
\item
Let $V$ be a finite-dimensional graded vector space:
\begin{enumerate}
\item
Let $\g=\Der_0\hat{S}\Pi^*V$ be the graded Lie algebra of formal
vector fields on $V$ with vanishing constant term and $m\in
\MC(\g,\ground)$ be an MC-element in $\g$ giving $V$ the structure of
an $L_\infty$-algebra. Then there is a map of dg coalgebras
${\chi}:\CE_*(V)\ra\CE_*(\g)$ determined by a collection of maps
$\chi_r:(\Pi V)^{\otimes r}\ra\Pi \g, r=1,2,\ldots$ such that
\[
\chi_r(w_1,\ldots, w_r)(x_1,\ldots x_n)=m_{n+r}(w_1,\ldots, w_r,
x_1,\ldots, x_n).
\]
where $w_i, x_i\in \Pi V$; $r\geq 0$ and $n\geq 1$.
\item
Let $\h=\Der_0\hat{T}\Pi^*V$ be the graded Lie algebra of formal
noncommutative vector fields on $V$ with vanishing constant term and
$m\in \MC(\h,\ground)$ be an MC-element in $\h$ giving $V$ the
structure of an $A_\infty$-algebra. Then there is a map of dg
coalgebras ${\chi}:\CE_*(V)\ra\CE_*(\h)$ determined by a collection
of maps $\chi_r:(\Pi V)^{\otimes r}\ra\Pi \h, r=1,2,\ldots$ such that
\[
\chi_r(w_1,\ldots, w_r)(x_1,\ldots x_n)=
\sum_{\sigma\in S_r}\frac{1}{r!}
m_{n+r}(\sigma(w_1\otimes\ldots\otimes w_r)|
x_1,\ldots, x_n).
\]
where $w_i, x_i\in \Pi V$; $r\geq 0$ and $n\geq 1$.
\end{enumerate}
\item
Let $V$ be a finite-dimensional graded vector space with an inner
product (and thus, $\Pi V$ is a symplectic vector space):
\begin{enumerate}
\item
Let $\tilde\g=\Der_0\hat{S}\Pi^*V$ be the graded Lie algebra of formal
symplectic vector fields on $V$ with vanishing constant term and
$m\in \MC(\tilde\g,\ground)$ be an MC-element in $\tilde\g$ giving $V$ the
structure of a cyclic $L_\infty$-algebra. Then there is a map of dg
coalgebras ${\chi}:\CE_*(V)\ra\CE_*(\tilde\g)$ determined by a collection
of maps $\chi_r:(\Pi V)^{\otimes r}\ra\Pi \tilde\g, r=1,2,\ldots$ such that
\[
\chi_r(w_1,\ldots, w_r)(x_1,\ldots x_n)=m_{n+r}(w_1,\ldots, w_r,
x_1,\ldots, x_n).
\]
where $w_i, x_i\in \Pi V$; $r\geq 0$ and $n\geq 1$.
\item
Let $\tilde\h=\Der_0\hat{T}\Pi^*V$ be the graded Lie algebra of formal
noncommutative symplectic vector fields on $V$ with vanishing
constant term and $m\in \MC(\tilde\h,\ground)$ be an MC-element in $\tilde\h$
giving $M$ the structure of an $A_\infty$-algebra. Then there is a
map of dg coalgebras ${\chi}:\CE_*(V)\ra\CE_*(\tilde\h)$ determined by a
collection of maps $\chi_r:(\Pi V)^{\otimes r}\ra\Pi \tilde\h,
r=1,2,\ldots$ such that
\[
\chi_r(w_1,\ldots, w_r)(x_1,\ldots x_n)=
\sum_{\sigma\in S_r}\frac{1}{r!}
m_{n+r}(\sigma(w_1\otimes\ldots\otimes w_r)|
x_1,\ldots, x_n).
\]
where $w_i, x_i\in \Pi V$; $r\geq 0$ and $n\geq 1$.
\end{enumerate}
\end{enumerate}
\end{prop}
\noproof
\begin{rem}
The statement of Proposition \ref{curved} appears almost
tautologically identical to that of Theorem \ref{Lmap}; however the
essential difference between these results is that we allow the value
$r=0$ in the formulas for $\chi$ (but not in the formulas for
${\mathbf f}$ (\ref{Linf1}), (\ref{Linf2}) ). This has the effect
that $\chi$ \emph{does not} give an $L_\infty$-map from $V$ into
$\g$. In fact, $\chi$ could be thought of as a \emph{curved}
$L_\infty$-map, cf. the last section concerning curved
$L_\infty$-algebras. However, having such a map is sufficient to get
an induced map on the CE homology.
\end{rem}
The above result is particularly interesting when $V$ is a cyclic
$L_\infty$- or $A_\infty$-algebra. In that case letting the dimension
of $V$ go to infinity we obtain the stable Lie algebra $\g_{\infty}$
whose homology $\CE(\g_{\infty})$ is essentially Kontsevich's graph
complex in the $L_\infty$-case and the complex of ribbon graphs in
the $A_\infty$-case, \cite{Kon, Ham, HaL}. Therefore, we obtain a map
from $\CE_*(V)$ to the appropriate version of the graph complex. This
map is related to, but quite different from another construction, the
so-called direct construction of Kontsevich which associates a graph
homology class to $V$ itself. The precise connection between these constructions
is the subject of a forthcoming paper.

\section{Homotopy fiber sequences of $L_\infty$-algebras}
In the previous section we constructed an $L_\infty$-map ${\mathbf
f}:V\ra\overline{\CE}(V,V)$ for an $L_\infty$-algebra $V$ which fits
into a homotopy fiber sequence of dg vector spaces
\[\xymatrix{
V\ar^-{\bf f}[r]&\overline{\CE}(V,V)\ar[r]&\CE(V,V),
}\]
and similarly in the $A_\infty$-algebra case. It is natural to ask
whether these are actually homotopy fiber sequences \emph{in the
category of $L_\infty$-algebras}. In this section we show that this is indeed the
case and explain the relation to work of Voronov and of Fiorenza-Manetti.

We begin by recalling the relevant notions. Note first that for any $L_\infty$-algebra
$(V,m)$ there is defined another such
$V[z,dz]:=V\otimes\ground[z,dz]$ where $\ground[z,dz]$ is the
polynomial de Rham algebra on the interval; the evaluation maps
$|_{z=a}:\ground[z,dz]\ra\ground, a=0,1$,  induce the
corresponding evaluation maps $|_{z=a}:V[z,dz]\ra V$.

\begin{defi}
Let $V, W$ be two $L_\infty$-algebras and $f, g:V\ra W$ be two maps
between them. A \emph{homotopy} between $f$ and $g$ is an
$L_\infty$-algebra map $s:V\ra W[z,dz]$ such that $s|_{z=0}=f$ and
$s|_{z=1}=g$.
\end{defi}
Note that a map $V\ra W$ is by definition a (formal) cdga map
$\hat{S}\Pi W^*\ra\hat{S}\Pi V^*$; the latter is equivalent to an
element in  $\MC(W, \hat{S}\Pi V^*)$. We shall be particularly
interested in the situation when $W$ is a dgla in which case
$\tp{W}{(\hat{S}\Pi V^*)_+}$ is a (formal) dgla. Then there is the
following characterization of homotopy:
\begin{prop}\label{equi}
Let $V$ be an $L_\infty$-algebra, $W$ be a dgla and $f,g\in \MC(W,
\hat{S}\Pi V^*)$ be two corresponding $L_\infty$-maps. Then $f$ and
$g$ are homotopic if and only if there exists an element $\xi\in
\tp{\Pi W}{(\hat{S}\Pi V^*)_+}$ such that $f=e^\xi ge^{-\xi}$.
\end{prop}
\begin{proof}
This follows from a well-known characterization of homotopy in formal
(or pronilpotent) dglas, cf. for example \cite{CL}, Theorem 4.4.
\end{proof}

Furthermore we could form  the homotopy category of dglas by formally
inverting quasi-isomorphisms. There is also the notion  of a homotopy
fiber sequence.
\begin{defi}
For a dgla map $f:\mathfrak{g\ra h}$ the \emph{homotopy fiber} of $f$
is defined as $H_f:=\{(g,h)\in \mathfrak g\oplus\mathfrak
h[z,dz]:h|_{z=0}=0, h|_{z=1}=f(g)\}$; the projection onto $\mathfrak
g$ is a dgla map $H_f\ra \mathfrak g$. The sequence
$H_f\ra\mathfrak{g\ra h}$ as well as any $3$-term sequence of dg Lie
algebras homotopy equivalent to it will be called a \emph{homotopy
fiber sequence} of dglas.
\end{defi}
For a homotopy fiber sequence
$\xymatrix{H_f\ar[r]&\mathfrak{g}\ar^f[r]&\mathfrak h}$ the composite
map $H_f\ra \mathfrak h$ is homotopic to zero; moreover it possesses
a semi-universal property: any map $\phi:\mathfrak{a\ra h}$ of dglas
which is homotopic to zero after composing with $f$ factors through
$H_f$. Namely, if $s:\mathfrak{a}\ra\mathfrak{h}[z,dz]$ is a
nullhomotopy of $f\circ\phi$, then the factorization is given by
$(\phi,s):\mathfrak{a}\to H_f\subset \mathfrak g\oplus\mathfrak
h[z,dz]$. We will need a slight extension of this property: suppose
that $\mathfrak{a}$ is an $L_\infty$-algebra and $\phi$ is an
$L_\infty$-map from $\mathfrak{a}$ into $\mathfrak{g}$ whose
composition with $f$ is nullhomotopic. Then the same formula as above
gives an $L_\infty$-map from $\mathfrak{a}$ into $H_f$.

The dgla homotopy fiber $H_f$ should be compared with the dg homotopy
fiber $C_f=\mathfrak{g}\oplus\Pi\mathfrak{h}$ with differential $d$
defined as follows. For $(g,h)\in \mathfrak{g}\oplus\Pi\mathfrak{h}$,
set $d(g,h)=(\partial g, -\partial h + f(g))$.

The following result is standard (see, e.g., \cite[Lemma 3.2]{FM}).
\begin{prop}\label{standard}
The dg vector space $C_f$ is a retract of $H_f$, via the maps
$i:C_f\ra H_f$ and $\pi:H_f\ra C_f$, given by
$$i(g,h)=(g, f(g)z + h dz);$$
$$\pi(g,h(z,dz))=(g,\int_0^1 h(z,dz)).$$
\end{prop}
\noproof

\begin{rem}
Note that dg fiber $C_f$ does not itself support a dgla structure.
However Fiorenza and Manetti \cite{FM} use the retraction in
Proposition~\ref{standard} to transport the dgla structure of $H_f$
to an $L_\infty$-algebra structure on $C_f$.
\end{rem}

Note that the category of dglas is a subcategory of the category of
$L_\infty$-algebras, and this inclusion induces an equivalence of
homotopy categories obtained by inverting respectively strict or
$L_\infty$-quasi-isomorphisms. We can now define a homotopy fiber
sequence of $L_\infty$-algebras.
\begin{defi}
The sequence of $L_\infty$-algebras and maps
\[
\mathfrak{a\ra g\ra h}
\]
is called \emph{homotopy fiber} if it is $L_\infty$-quasi-isomorphic
to a homotopy fiber sequence of dglas.
\end{defi}
It is clear that the `homotopy fiber' of an $L_\infty$-map
$f:\mathfrak g\ra \mathfrak h$ (i.e. an $L_\infty$-algebra $\mathfrak
a$ which can be included in a fiber sequence $\xymatrix{\mathfrak
a\ar[r]&\mathfrak{g}\ar^f[r]&\mathfrak h}$) is determined uniquely up
to a non-canonical isomorphism in the homotopy category and that it
possesses the analogous semi-universal property as the homotopy fiber
of a dgla map.

\begin{theorem}\label{homf}\
\begin{enumerate}
\item
\begin{enumerate}
\item Let $V$ be an $L_\infty$-algebra; then the sequence
$$\xymatrix{V\ar^-{\mathbf f}[r]&\overline{\CE}(V,V)\ar^-g[r]&\CE(V,V))}$$ is a
fiber sequence of $L_\infty$-algebras.
\item Let $V$ be a cyclic $L_\infty$-algebra; then the sequence
$$\xymatrix{V\ar^-{\mathbf f}[r]&\overline{\CE}(V,V)\ar^-g[r]&\CE(V,V))}$$ is a
fiber sequence of $L_\infty$-algebras.
\end{enumerate}
\item
\begin{enumerate}
\item
Let $V$ be an $A_\infty$-algebra; then the sequence
$$\xymatrix{V\ar^-{\mathbf f}[r]&\overline{\Hoch(V,V)}\ar^-g[r]&\Hoch(V,V)}$$ is a
fiber sequence of $L_\infty$-algebras.
\item Let $V$ be a cyclic $A_\infty$-algebra; then the sequence
$$\xymatrix{V\ar^-{\mathbf f}[r]&\overline{\CE}(V,V)\ar^-g[r]&\CE(V,V)}$$ is a
fiber sequence of $L_\infty$-algebras.
 \end{enumerate}
\end{enumerate}
\end{theorem}
\begin{proof}
We restrict ourselves with proving the statement about
$L_\infty$-algebras; the other parts are proved similarly. Denote by
$\alpha=g\circ {\mathbf f}: V\ra\CE(V,V)$ the composite
$L_\infty$-map.
 Let
us first prove that $\alpha$ is homotopic to zero. To this end we
view $\alpha$ as an MC element in the dgla $\hat{S}\Pi
V^*\otimes\Pi\CE(V,V)$. The zero map is represented by the CE
differential $1\otimes[,m]$ in this dgla.  Observe that $\alpha=e^\xi
d e^{-\xi}$ where $\xi=\sum w^i\otimes w_i$ where $w_i$ and $w^i$ are
dual bases in $\Pi V$ and $\Pi V^*$ respectively, and $w_i$ is viewed
as a (constant) derivation of $\hat{S}\Pi V^*$. It follows from
Proposition \ref{equi} that $\alpha$ is indeed homotopic to zero. The
homotopy has the following explicit form: $\alpha_z=e^{z\xi} d
e^{-z\xi}+\xi dz$.

The arguments similar to those used to derive formulas (\ref{Linf1})
translate this homotopy to an $L_\infty$-map $s=\{s_r\}$ from $V$ to
$\CE^*(V,V)[z, dz]$:
\[
s_r(w_1,\ldots, w_r)=\begin{cases}{z^rf_r(w_1, \ldots, w_r), \text {
if $r>1$}}\\{zf_1(w)+\hat{w}dz, \text{ if $r=1$}}\end{cases}
\]
Here $\hat{w}$ refers to the constant derivation of $\hat{S}\Pi
V^*$ associated to $w$ viewed as an element in $\Pi\CE(V,V)$.

The $L_\infty$-map $s$ gives rise to an $L_\infty$-map
$\tilde{s}=\{\tilde{s}_r\}$ from $V$ to $H_g\subset
\overline{\CE}^*(V,V)\oplus\Pi \CE^*(V,V)$. Here $\tilde{s}_r:(\Pi
V)^{\otimes n}\ra \Pi H_g$:
\[
\tilde{s}_r(w_1,\ldots, w_r)=(f_r(w_1,\ldots, w_r),s_r(w_1,\ldots,
w_r)).
\]

It remains to show that $\tilde{s}_1:\Pi V \ra \Pi H_g$ is a
quasi-isomorphism. For this, it suffices to prove that the composite
map $\pi\circ\tilde{s}_1: \Pi V \to \Pi C_g$ is a quasi-isomorphism.
We have
$$(\pi\circ\tilde{s}_1)(w) = (f_1(w),\int_0^1 s_1(w)) =
(f_1(w),\hat{w}).$$

Observe that $\CE(V,V)$ is the dg mapping cone of
$f_1:V\to\overline{\CE}(V,V)$, and $g$ is the canonical map of
$\overline{\CE}(V,V)$ into that cone. So we conclude the argument by
appealing to the following general fact.
\begin{lem}
Let
\[\xymatrix{X\ar^-f[r]&Y\ar^-g[r]&\operatorname{cone}(f)}\]
be a standard cofiber sequence of dg vector spaces and let $C_g$ be
the homotopy fiber of $g$; recall that $C_g$ has $Y\oplus\Pi Y \oplus
X$ as its underlying graded vector space. Then the map $h:X\to C_g$
defined by $h(x)=(f(x),0,x)$ is a quasi-isomorphism.
\end{lem}
\end{proof}

We now describe the relationship of Theorem \ref{homf} to Voronov's work
on higher derived brackets \cite{Vor,Vor'}. We are indebted to the referee
for bringing the precise connection to our attention. Voronov's starting point is
an inclusion $K\hookrightarrow L$ of dgla's with abelian complement $W$.
He shows that $L\oplus \Pi W$ may be endowed with an $L_\infty$-algebra structure
so that the projection map $L\oplus \Pi W$ is a replacement of the original
inclusion $K\hookrightarrow L$ by a fibration. This implies that
$\Pi W$, equipped with Voronov's higher derived brackets, is a homotopy
fiber of $K\hookrightarrow L$.

The inclusion $\overline{\CE}(V,V)\hookrightarrow {\CE}(V,V)$ (and its variations) we have considered in this section is a particular case of Voronov's setup, with complementary space $W=\Pi V$. His theory, applied in this case, shows that
the $L_\infty$-algebra $(V,m)$ is a homotopy fiber of the inclusion. The new statement in Theorem \ref{homf} is the assertion that one obtains a
homotopy fiber sequence $V\to\overline{\CE}(V,V)\to {\CE}(V,V)$, in which
the first map is the map $\bf$ explicitly described in formulas (\ref{Linf1})
and (\ref{Linf2}).

Our explicit realisation of a homotopy fiber sequence may be compared with
the main result of Fiorenza and Manetti \cite{FM}.
Given an arbitrary dgla map $K\to L$, they construct a homotopy fiber sequence
$C\ra K \ra L$ of $L_\infty$-algebras, where the underlying dg vector space
of the homotopy fiber $L_\infty$-algebra $C$ is the homotopy fiber
of $K\to L$ \emph{in the category of dg vector spaces}.

\subsection{Curved $L_\infty$-algebras and deformations} We will now
outline one application of the above theorem which, roughly, could be
formulated as saying that an $L_\infty$-algebra $V$ controls the
deformations of $V$ as an $L_\infty$-algebra \emph{which are trivial}
as deformations of curved $L_\infty$-algebras. A similar statement
can, of course, be made in the $A_\infty$ or cyclic contexts; we will restrict ourselves to treating the $L_\infty$ case
only.
\begin{defi}
A curved $L_\infty$-algebra structure on a graded vector space $V$ is
a (continuous) odd derivation $m$ of the algebra $\hat{S}\Pi V^*$
such that $[m,m]=0$.
\end{defi}
Note that such a derivation can be written as $m=m_0+m_1+\ldots$
where $m_i$ is an odd map $\hat{S}(\Pi V^*)\ra \Pi V^*$ which is
equivalent to a collection of $S_n$-invariant maps $(\Pi V)^{\otimes
n}\ra \Pi V$ subject to some constraints coming from the identity
$[m,m]=0$. The difference with the notion of an $L_\infty$-algebra is
that $m$ is not required to have zero constant term; consequently
$m_0$ could be non-zero and $m_1$ does not necessarily square to
zero.

For a curved $L_\infty$-algebra $(V,m)$ we could form its
Chevalley-Eilenberg complex $\CE(V,V)$; its underlying space is
$\Der\hat{S}(V)$ and the differential is defined as the commutator
with $m$, just as in the uncurved case. The condition $[m,m]=0$
ensures that the differential squares to zero, however there is no
analogue of $\overline{\CE}(V,V)$ since $[,m]$ does not restrict to
the subspace of derivations with the vanishing constant term.
Clearly, $\CE(V,V)$ is a dgla. This dgla controls
deformations of $V$; let us briefly recall the general setup of
(derived) deformation theory, cf. for example \cite{Me}.

Let $\mathfrak g$ be a dgla (or, more generally, an
$L_\infty$-algebra). Let $\Def_{\mathfrak g}$ be the functor
associating to a formal cdga $A$ the pointed set
$\MCmod(\g,A)$
 of MC-elements in
$\tp{{\mathfrak g}}{A_+}$ \emph{modulo gauge equivalence}. Here we
say that two elements $x_0, x_1\in\MC({\mathfrak g},A)$ are gauge
equivalent if they are homotopic, i.e. there exists an MC element
$X\in\MC({\mathfrak g}[z,dz], A)$ such that $X|_{z=0}=x_0$ and
$X|_{z=1}=x_1$. The base point in $\Def_{\mathfrak g}$ is the gauge
class of the zero MC element. The functor $\Def_{\mathfrak g}$ is
homotopy invariant in ${\mathfrak g}$, which means that a
quasi-isomorphism ${\mathfrak g}\ra {\mathfrak g}^\prime$ induces a
natural bijection $\Def_{\mathfrak g}(A)\cong \Def_{{\mathfrak
g}^\prime}(A)$.
\begin{defi}
A deformation of an $L_\infty$-algebra $V$ over a formal cdga $A$ is
an MC-element in $\tp{\overline{\CE}(V,V)}{A_+}$. Two such
deformations are called \emph{equivalent} if the corresponding
MC-elements are gauge equivalent. Similarly a deformation of an
curved $L_\infty$-algebra $V$ over a formal cdga $A$ is an MC-element
in $\tp{{\CE}(V,V)}{A_+}$. Two such deformations are called
\emph{equivalent} if the corresponding MC-elements are gauge
equivalent.
\end{defi}
Now let ${\mathfrak g}\ra {\mathfrak h}\ra {\mathfrak a}$ be a
homotopy fiber sequence of $L_\infty$-algebras. Associated to it
is the sequence of functors $\Def_{\mathfrak g}\ra\Def_{\mathfrak
h}\ra\Def_{\mathfrak a}$ which is fiber in the sense that the
preimage in $\Def_{\mathfrak h}$ of the base point in $\Def_{\mathfrak a}$
 lies in the image of $\Def_{\mathfrak g}$. This
could be easily seen by using the homotopy invariance of $\Def$ and
reducing to a short exact sequence of dglas. Note that the map
$\Def_{\mathfrak g}\ra\Def_{\mathfrak h}$ is not necessarily
injective; in fact it has an action of the group of homotopy
self-equivalences of the functor $\Def_{\mathfrak a}$, however
explaining this would take us too far afield and we refrain from
giving further details.

We obtain the following result, which is an immediate consequence of
theorem \ref{homf}.
\begin{cor}
A deformation of an $L_\infty$-algebra $V$ over a formal cdga $A$
which is trivial as its deformation as a curved $L_\infty$-algebra is
always given by a twisting with some MC-element $\xi\in\MC(V,A)$.
\end{cor}

In conclusion note that the constructions of this paper have an
analogue in conventional homotopy theory. Let $X$ be a pointed space;
denote by $\Map(X,X)$ and $\overline{\Map}(X,X)$ the space of all
(respectively pointed) self-maps of $X$. There is an obvious fiber
sequence of spaces
\[
\overline{\Map}(X,X)\ra\Map(X,X)\ra X
\]
which gives rise to the homotopy  fiber sequence
\[
\Omega(X)\ra\overline{\Map}(X,X)\ra\Map(X,X)
\]
where $\Omega(X)$ is the space of based loops on $X$. This is
parallel to any of the fiber sequences (\ref{fibL1}), (\ref{fibL2}),
(\ref{fibL3}) or (\ref{fibL4}). Note that there does not seem to be a
simple description of the map $\Omega(X)\ra\overline{\Map}(X,X)$.

\end{document}